\documentclass[a4paper]{amsart}

\usepackage{amsmath,amssymb,amsthm,amsfonts}
\usepackage[shortlabels]{enumitem}
\usepackage[UKenglish]{babel}
\usepackage{microtype}
\usepackage[colorlinks=false, pdfborder={0 0 0}]{hyperref}

\newcommand{\StS}{\mathcal{S}\mathcal{S}}
\newcommand{\eqnorm}[1]{\left|\left|\left|#1\right|\right|\right|}
\newcommand{\norm}[1]{\left\lVert#1\right\rVert}
\newcommand{\conv}{\operatorname{conv}}
\newcommand{\diam}{\operatorname{diam}}
\newcommand{\Dz}{\operatorname{Dz}}
\newcommand{\m}{\operatorname{m}}

\newtheorem{theorem}{Theorem}[section]
\newtheorem{proposition}[theorem]{Proposition}
\newtheorem{corollary}[theorem]{Corollary}
\newtheorem{lemma}[theorem]{Lemma}
\theoremstyle{definition}
\newtheorem{definition}[theorem]{Definition}

\begin{document}

\title{Maps with the Radon\textendash{}Nikod\'ym property}

\dedicatory{Dedicated to the memory of Jonathan Borwein.}

\author{L. Garc\'ia-Lirola}
\address{Departamento de Matem\'aticas, Facultad de Matem\'aticas, Universidad de Murcia, 
30100 Espinardo (Murcia), Spain}
\email{luiscarlos.garcia@um.es}
\author{M. Raja}
\address{Departamento de Matem\'aticas, Facultad de Matem\'aticas, Universidad de Murcia, 
30100 Espinardo (Murcia), Spain}
\email{matias@um.es}

\thanks{Partially supported by the grants MINECO/FEDER MTM2014-57838-C2-1-P and Fundaci\'on S\'eneca CARM
19368/PI/14}

\date{December, 2016}



\begin{abstract}
We study dentable maps from a closed convex subset of a Banach space into a metric space as an attempt of generalize the Radon-Nikod\'ym property to a ``less linear'' frame. We note that a certain part of the theory can be developed in rather great generality. Indeed, we establish that the elements of the dual which are  ``strongly slicing'' for a given uniformly continuous dentable function form a dense $\mathcal G_\delta$ subset of the dual. As a consequence, the space of uniformly continuous dentable maps from a closed convex bounded set to a Banach space is a Banach space. However some interesting applications, as Stegall's variational principle, are no longer true beyond the usual hypotheses, sending us back to the classical case. Moreover, we study the connection between dentability and approximation by delta-convex functions for uniformly continuous functions. Finally, we show that the dentability of a set is closely related with the dentability of delta-convex maps defined on it.
\end{abstract}

\maketitle

\section{Introduction}\label{sec:1}

The Radon\textendash{}Nikod\'ym property (RNP) plays a central role in Banach space theory, particularly in isomorphic and nonlinear theories. It is related to the differentiation of Lipschitz maps (Aronszajn\textendash{}Christensen\textendash{}Mankiewicz's theorem), the extremal structure of convex sets (exposed points), representation theory without compactness (Edgar's theorem), representation of dual function spaces, optimization theory (Stegall's variational principle), etc. The interested reader in RNP, theory and applications, is addressed to~\cite{BL00,B83,D75,FHHMZ11}.

The RNP was extended to linear operators by Re\u\i nov~\cite{R75} and Linde~\cite{L76}. In order to propose a definition for nonlinear maps, we will consider the most geometrical one among the characterisations of the RNP. Our starting point is to consider the notion of \emph{dentability} for maps, which appears as a strengthening of the point of continuity property with the help of the geometry on the domain. Let ${\mathbb H}$ denote the set of all the open half-spaces of a Banach space $X$\@.

\begin{definition}
Let $C \subset X$ be a nonempty subset of a Banach space $X$ and let $M$ be a metric space. A map $f\colon C \rightarrow M$ is said to be \emph{dentable} if for every nonempty bounded set $A \subset C$ and $\varepsilon>0$, there is $H \in {\mathbb H}$ such that $A \cap H \not = \emptyset$ and $\diam(f(A \cap H)) < \varepsilon$\@.
\end{definition}

The set of dentable maps from $C$ to $M$ will be denoted ${\mathcal D}(C,M)$\@. Note that a closed convex set $C$ has the Radon\textendash{}Nikod\'ym property if and only if the identity map $\mathbb{I}\colon C \rightarrow (C,\norm{\cdot})$ is dentable thanks to Rieffel's characterization. Moreover, a bounded subset $A\subset X$ is Asplund if and only if the identity $\mathbb I\colon X^\ast \rightarrow (X^\ast, d_A)$ is dentable, where $d_A(x^*,y^*)=\sup\{|x^*(x)-y^*(x)|: x\in A\}$\@. In addition, a map $f\colon C \rightarrow (M,d)$ induces a pseudometric on $C$ by the formula $\rho(x,y)=d(f(x),f(y))$ and the dentability of the map $f$ is equivalent to the ordinary subset dentability of $C$ with respect to $\rho$\@. Nevertheless, we prefer to consider maps from $C$ to a metric space $M$ since $M$ can carry other structures, as algebraic ones. Let us remark that the notion of dentable map should be compared to two previous related concepts. First, $\sigma$-slicely continuous maps introduced in~\cite{MOTV09} provide a characterisation of the existence of an equivalent LUR norm in a Banach space. On the other hand, $\sigma$-fragmentable maps were introduced in~\cite{JOPV93} in order to study selection problems. 

By ${\mathcal D_U}(C,M)$ we denote the set of maps from ${\mathcal D}(C,M)$ which are moreover uniformly continuous on bounded subsets of $C$\@. That technical condition is necessary in order to perform several operations motivated by the geometrical study of the RNP property, which ensures nice properties for this class of maps. These properties are summarized in the next result.

\begin{theorem}\label{th:structure} Let $C \subset X$ be a closed convex set. If $M$ is a vector space, then ${\mathcal D_U}(C,M)$ is a vector space.
Assume moreover that $C$ is bounded. Then:
\begin{enumerate}[(a)]
\item if $M$ is a complete metric space, then ${\mathcal D_U}(C,M)$ is complete for the metric of uniform convergence on $C$;
\item if $M$ is a Banach space, then ${\mathcal D_U}(C,M)$ is a Banach space;
\item if $M$ is a Banach algebra (resp. lattice), then ${\mathcal D_U}(C,M)$ is a Banach algebra (resp. lattice).
\end{enumerate}
\end{theorem}

The key for the proof of Theorem~\ref{th:structure} is the fact that there are many functionals, in a categorical sense, defining slices of small oscillation, which will be proved in Theorem~\ref{th:ssGdelta}.

Particularly interesting is the case where $C$ is convex and closed and $M=\mathbb R$, because then ${\mathcal D}(C,{\mathbb R})$ contains the bounded above lower semicontinuous convex functions. Indeed, for $A \subset C$ and $\varepsilon>0$, any  $H \in {\mathbb H}$ containing a point of $A$ and disjoint from the convex closed set
\[ D = \{x \in C: f(x) \leq \sup\{f,A\} - \varepsilon \} \]
will satisfy that $\diam(f(A \cap H)) \leq \varepsilon $\@.
 Having in mind that $\mathcal{D}(C,\mathbb R)$ is a linear space, the difference of two bounded convex continuous functions is dentable. The differences of convex functions,
  usually named $\mathcal{DC}$ functions, play an important role in variational analysis and optimization (see, e.g.~\cite{BB11,H85,T95}). Moreover, the possibility of a real function to be uniformly approximated by $\mathcal{DC}$ functions is closely related to its dentability. We do not state here the main result Theorem~\ref{th:deltaconvex} because it requires special definitions. Instead of introducing more notions here, we will give a glance of the sort of results we can prove for dentable functions.

The following expresses the ``equi-dentability'' for finitely many dentable real functions. Unfortunately, our techniques require uniform continuity.

\begin{proposition}\label{prop:equident}
Let $C \subset X$ be a bounded closed convex set. Given $f_1, \dotsc, f_n \in {\mathcal D_U}(C,{\mathbb R})$ and $\varepsilon >0$, there is $H \in {\mathbb H}$ such that $C \cap H \not =\emptyset$ and
\[ \max\{ \diam(f_1(C \cap H)), \dotsc, \diam(f_n(C \cap H)) \} < \varepsilon. \]
\end{proposition}

Let state separately the following curious corollary, which was observed by Bourgain in~\cite{B80}.

\begin{corollary}\label{cor:equidual}
Let $C \subset X$ be a bounded closed convex set. Given $x^*_1,\dotsc,x^*_n \in X^*$ and $\varepsilon >0$, there is $H \in {\mathbb H}$ such that $C \cap H \not =\emptyset$ and
\[ \max\{ \diam(x^*_1(C \cap H)), \dotsc, \diam(x^*_n(C \cap H)) \} < \varepsilon. \]
\end{corollary}

Another particular case that we have considered is the dentability of the identity map when $M=C$ endowed with a metric which is uniformly continuous with respect to the norm. In that case, not much can be obtained unless the metric induces the norm topology. But those hypotheses are not more general than the RNP as the following result shows.

\begin{theorem}\label{th:rnp}
Let $C$ be a closed convex subset which is dentable with respect to a complete metric $d$ defined on it. Assume moreover that $d$ is uniformly continuous on bounded sets with respect to the norm and induces the norm topology. Then $C$ has the RNP.
\end{theorem}

This paper is organised as follows. In the second section we establish stability properties of dentable maps, which allow us to prove Theorem~\ref{th:structure}. The third section includes a characterisation of sets with the RNP in terms of dentability of continuous maps defined on them, and a characterisation of uniformly continuous finitely dentable maps. The fourth section is devoted to the relation between dentable maps and delta-convex maps. Finally, in the fifth section we investigate sets which are dentable with respect to a metric defined on it, including the proof of Theorem~\ref{th:rnp}. Throughout the paper $C$ will denote a closed convex subset of a Banach space $X$ and $M$ will denote a metric space with a metric $d$\@. Our notation is standard and will normally follow the books~\cite{FHHMZ11} and~\cite{BW10}.

Along the paper we will consider only real Banach spaces. We mention in passing that Ghoussoub and Maurey studied in~\cite{GM89} the geometrical structure of sets with the \emph{analytic Radon\textendash{}Nikod\'ym property}, which may be viewed as a complex analogue of the RNP.

\section{Properties of the dentable maps}\label{sec:2}

We begin the section by studying the relation between dentable maps and RN-operators. We need the following result, which should be compared with~\cite[Proposition 2.3.2]{B83}.

\begin{proposition}\label{prop:dent} A map $f\colon C\rightarrow M$ is dentable if and only if for every nonempty bounded set $A\subset C$ and every $\varepsilon>0$ there exists $x\in A$ such that $x\notin \overline{\conv}(A\setminus f^{-1}(B_M(f(x),\varepsilon))$\@.
\end{proposition}

\begin{proof}
First assume that $f$ is a dentable map. Fix $\varepsilon>0$ and $A\subset C$ nonempty and bounded. By hypothesis, there exists an open half-space $H$ such that $A\cap H\neq \emptyset$ and $\diam(f(A\cap H))<\varepsilon$\@. Then $A\cap H\subset f^{-1}(B_M(f(x),\varepsilon))$, so $\overline{\conv}(A\setminus f^{-1}(B_M(f(x),\varepsilon))\cap H=\emptyset$ and any $x\in A\cap H$ does the work.

Conversely, fix $\varepsilon>0$ and let $A\subset C$ be nonempty and bounded. Take $x\in A$ so that $x\notin \overline{\conv}(A\setminus f^{-1}(B_M(f(x),\varepsilon/2))$\@. Then the dentability condition is witnessed by any slice of $A$ separating $x$ from $\overline{\conv}(A\setminus f^{-1}(B_M(f(x),\varepsilon/2))$\@. 
\end{proof}

Re\u\i nov~\cite{R77} characterised RN-operators as those bounded operators satisfying the condition in Proposition~\ref{prop:dent}. Therefore, the notion of dentable function extends the class of RN-operators to the non-linear setting.

\begin{corollary} Let $T\colon X\rightarrow Y$ be a bounded linear operator between Banach spaces. Then $T$ is a RN-operator if, and only if, $T$ is dentable.
\end{corollary}

Our next goal is to establish that there are many functionals defining slices of small oscillation for a dentable map. The basic result is a version of the superlemma of Asplund\textendash{}Bourgain\textendash{}Namioka as presented in~\cite[Theorem 3.4.1]{B83}.

\begin{lemma} \label{lemma:super}
Let $f\colon C \rightarrow M$ be a uniformly continuous map and $A, B \subset X$ be bounded closed convex subsets such that $A \subset C \subset
\overline{\conv}(A \cup B)$, $\diam(f(A)) < \varepsilon$ and $A \setminus B \not = \emptyset$\@. Then there exists  $H \in {\mathbb H}$ such that $A\cap H  \not = \emptyset$, $B\cap H = \emptyset$ and $\diam(f(C\cap H)) <\varepsilon$\@.
\end{lemma}

\begin{proof} Take $\eta = 3^{-1} (\varepsilon - \diam(f(A)) )$\@. Let $\delta>0$ be such that if $x,y \in C$ satisfy $\|x-y\|<\delta$ then $d(f(x),f(y)) < \eta$\@.
Consider for some $r \in (0,1]$ the convex set
\[ D_r = \{ (1-\lambda) y + \lambda z: y \in A, z \in B, \lambda \in [r,1] \}. \]

First we claim that $A\setminus \overline{D_r}$ is nonempty. Indeed, let $x^*\in X^*$ be such that $\sup\{x^*, B\}<\sup\{x^*,C\}$, which exists since $A\setminus B$ is nonempty. Then $\sup\{x^*,A\}=\sup\{x^*,C\}$ and thus
\[ \sup\{x^*, \overline{D_r}\} = \sup\{x^*, D_r\} \leq (1-r)\sup\{x^*,A\} + r \sup\{x^*, B\} <\sup\{x^*, A\}.\]
So $A\not\subset \overline{D_r}$\@.

Now note that for $x \in C\setminus D_r \subset \conv(A \cup B) \setminus D_r$ there are $y \in A$, $z \in B$ and $\lambda \in [0,r]$ such that $x=(1-\lambda)y + \lambda z$\@. Therefore, $\|x-y\| \leq r \|y-z\|$\@. If we take
$r \in (0,1]$ such that $r \, \diam(A-B) < \delta$, then $d(f(x),f(y)) <\eta$\@. That implies
\[\diam( f (C\setminus D_r ) ) < \diam(f(A))+ 2 \eta < \varepsilon.\]
Finally, any $H \in {\mathbb H}$ separating points of $A$ from $\overline{D_r}$ will satisfy that $\diam(f(C\cap H)) < \varepsilon$,
as desired. 
\end{proof}

The following result can be proved as the analogous for the RNP (see
\cite[Proposition 3.5.2]{B83}).

\begin{lemma} \label{lemma:slice}
Let $C, D \subset X$ be bounded closed convex subsets such that $C \setminus D \not = \emptyset$ and suppose that $f\colon C \rightarrow M$ is a uniformly continuous dentable map.
Given $\varepsilon>0$, there exists a half-space $H$ such that  $D\cap H =\emptyset$, $C\cap H \not =\emptyset$ and
$\diam(f(C\cap H)) <\varepsilon$\@.
\end{lemma}

\begin{proof} Take $E =  \overline{\conv}(C \cup D)$ and let
\[F= \{x \in E: \textrm{ there is an } x^* \in X^* \textrm{ such that } x^*(x)=\sup\{x^*,E\} > \sup\{x^*,D\}\}\,.\]
Note that Bishop\textendash{}Phelps theorem ensures the existence of sup-attaining functionals arbitrarily close to functionals separating points of $C$ from $D$, so $F\neq \emptyset$\@. Moreover, $F \subset C$ and $E=\overline{\conv}(F \cup D)$\@. Indeed, the first part is standard (anyway, see~\cite[Theorem 3.5.1]{B83}). If the second inclusion does not hold, separation with Hahn\textendash{}Banach and Bishop\textendash{}Phelps again lead to a new point of $F$ outside from $F$, a contradiction.

Now, find a nonempty open slice $S$ of $\overline{\conv}(F)$ such that $\diam(f(\overline{S})) <\varepsilon$\@. Consider
$B=\overline{\conv}(D \cup (\overline{\conv}(F)\setminus S))$\@. We claim that $S \setminus B \not =\emptyset$\@. Indeed, suppose that
$S \subset B$, which clearly implies that $B=E$\@.
There are $x \in S$ and $x^* \in X^*$ such that $x^*(x)=\sup\{x^*,E\} > \sup\{x^*,D\}$\@. Since we are assuming that
\[E=\overline{\conv}(D \cup (\overline{\conv}(F) \setminus S)),\]
we should have $x \in \overline{\conv}(F) \setminus S$, which is impossible.

Finally, Lemma~\ref{lemma:super} with $A =\overline{S}$ and the same set $B$ provides a half-space $H$ which does not meet $D$ and such that $\diam(f(C \cap H)) < \varepsilon$\@. 
\end{proof}

It is convenient to introduce a specific notation for slices of a set $A\subset X$, namely
\[ S(A,x^*,t)  = \{ x \in A: x^*(x) > \sup\{x^*,A\} - t \}\,, \]
where $x^* \in X^*$ and $t>0$\@.

\begin{definition}
Let $f\colon C \rightarrow M$ be a map, $A$ be a bounded subset of $C$ and $x^* \in X^*$\@. We say that $A$ is $f$\emph{-strongly sliced} by $x^*$ if
\[ \lim_{t \rightarrow 0^+} \diam(f(S(A,x^*,t))) = 0\,, \]
and in such a case we say that $x^*$ is $f$\emph{-strongly slicing} on $A$\@. The set of all the $f$-strongly slicing functionals on $A$ will be denoted $\StS(f,A)$\@.
\end{definition}

Note that the notion of strongly slicing functional is similar to that of strongly exposing. However, a strongly exposing functional is always referred to a point of the set. Here is not the case since, in general, the slices are not converging to a point. That pathology will be studied further in relation with the dentability of sets.

We will need the following lemma.

\begin{lemma}[Lemma 3.3.3 of~\cite{B83}]\label{lemma:bour}
Suppose that $x^* \in X^*$ and $\|x^*\|=1$\@. For $r>0$ denote by $V_r$ the set $rB_X \cap \ker x^*$\@. Assume
that $x_0$ and $y$ are points of $X$ such that $x^*(x_0)>x^*(y)$ and $\|x_0-y\| \leq r/2$\@. If $y^* \in X^*$ satisfies
that $\|y^*\|=1$ and $y^*(x_0) > \sup \{y^*,y+V_r\}$, then $\|x^*-y^*\| \leq \frac{2}{r} \|x_0-y\|$\@.
\end{lemma}

If $C\subset X$ has the RNP, then the strongly exposing functionals form a dense $\mathcal G_\delta$ subset of $X^*$ (see~\cite[Theorem 3.5.4]{B83}). The next result establishes that an analogous statement holds for dentable maps. 

\begin{theorem} \label{th:ssGdelta}
If $f \in {\mathcal D_U}(C,M)$, then $\StS(f,A)$ is a dense ${\mathcal G}_\delta$ subset of $X^*$ for any nonempty bounded $A \subset C$\@.
\end{theorem}

\begin{proof}
For $n\in \mathbb N$ consider the set
\[ U_n = \{x^*\in X^* : \textrm{ there is } t>0 \textrm{ such that } \diam(f(S(A,x^*,t)))<1/n\}. \]
It is not difficult to see that $U_n$ is open. In order to prove that it is also dense, take $x^*\in X^*$ and $0<\varepsilon<1$\@. Pick $x_0\in A$ and $y\in X$ with $x^*(x_0)>a>x^*(y)$ for some $a\in \mathbb R$\@. Now take $r=2\varepsilon^{-1}\sup\{\norm{x-y}:x\in A\}$ and consider
\[ D= \overline{\conv}(A\cup(y+V_r))\cap\{x\in X: x^*(x)\leq a\},\]
where $V_r=rB_X\cap \ker x^*$\@.

Note that $x_0\notin D$, so $\overline{\conv}(A)\setminus D\neq \emptyset$\@. Moreover, $f$ is uniformly continuous on the bounded set $\overline{\conv}(A)$\@. Thus Lemma~\ref{lemma:slice} provides an open half-space $H$ such that $D\cap H=\emptyset$, $\overline{\conv}(A)\cap H\neq \emptyset$ and $\diam(f(\overline{\conv}(A)\cap H))<1/n$\@. Then clearly $A\cap H\neq \emptyset$ and $\diam(f(A\cap H))<1/n$, so $y^*\in U_n$ for $y^*\in S_{X^*}$ being the functional which determines $H$\@. Finally, we will show that $\norm{x^*-y^*}<\varepsilon$\@. For this, take $x_1\in A\cap H$\@. It suffices to show that $y^*(x_1)>\sup\{y^*,y+V_r\}$ and apply Lemma~\ref{lemma:bour}. Notice that $\{x^*\in X^*: x^*(x)>a\}\cap (y+V_r) = \emptyset$ because $x^*(y)<a$\@. Moreover, since $D\cap H=\emptyset$ we have $x^*(x)>a$ for every $x\in (y+V_r)\cap H$\@. So $(y+V_r)\cap H=\emptyset$ and therefore  $y^*(x_1)>\sup\{y^*,y+V_r\}$, which proves that $\norm{x^*-y^*}<\varepsilon$ and finishes the proof of the density of $U_n$\@.

Finally, the set $\StS(f,A) = \bigcap_{n=1}^\infty U_n$ is dense in $X^*$ by Baire theorem, as we want.
\end{proof}

As a consequence we get several corollaries.

\begin{corollary}\label{cor:prodDent}
Let $M_i$ be metric spaces and $f_i  \in {\mathcal D_U}(C,M_i)$ for $i=1,\dotsc,n$\@. Assume that $A \subset C$ is nonempty and bounded. Given $\varepsilon >0$, there exists an open half-space $H \subset X$ such that $A \cap H \not =\emptyset$ and
\[ \max\{ \diam(f_i(A \cap H)): i=1,\dotsc, n\} < \varepsilon. \]
Hence, if we set $M=\prod_{i=1}^n M_i$ endowed with a standard product metric and $f=(f_1,\dotsc,f_n)$, then
$f \in {\mathcal D_U}(C,M)$\@.
\end{corollary}

\begin{proof}
The intersection $\bigcap_{i=1}^n \StS(f_i,A)$ is non-empty by Baire theorem and Theorem~\ref{th:ssGdelta}. Moreover, every element of $\bigcap_{i=1}^n \StS(f_i,A)$ provide slices satisfying the required property.
\end{proof}

Proposition~\ref{prop:equident} and Corollary~\ref{cor:equidual} follow directly from Corollary~\ref{cor:prodDent}.

It is clear that the composition of a dentable map and a uniformly continuous map is dentable. For compositions with continuous maps we have the following result.  

\begin{corollary}\label{cor:composition} Let $g\colon M\to N$ be a continuous map between metric spaces $M$ and $N$ and  $f\in \mathcal D_U(C,M)$\@. Assume that $M$ is complete. Then $g\circ f$ is dentable. 
\end{corollary}

\begin{proof}
Take $A\subset C$ nonempty and bounded and fix $\varepsilon>0$\@. Let $x^*$ be an $f$-strongly slicing functional on $A$, which exists by Theorem~\ref{th:ssGdelta}. Since $M$ is complete, there exists $y_0\in \bigcap_{t>0} \overline{f(S(A,x^*,t))}$\@. Now, take $\delta>0$ such that $d(g(y),g(y_0))<\varepsilon$ whenever $y\in M$ and $d(y,y_0)<\delta$\@. Then there is $t>0$ so that $\diam(f(S(A,x^*,t)))=\diam(\overline{f(S(A,x^*,t))})<\delta$\@. It follows that $\diam((g\circ f)(S(A,x^*,t))<2\varepsilon$\@. 
\end{proof}

\begin{lemma}\label{lemma:binop}
Let $*$ be a binary operation on $M$ which is uniformly continuous on bounded sets. Then $f*g \in {\mathcal D_U}(C,M)$ whenever
$f,g \in {\mathcal D_U}(C,M)$\@.
\end{lemma}

\begin{proof}
Note that a uniformly continuous function on a convex bounded set is bounded. This fact and the hypothesis on the operation $*$ imply that $f * g$ is uniformly continuous on bounded sets whenever $f,g$ are. Now, if $A \subset C$ is bounded and nonempty, find $x^* \in X^*$
such that it is simultaneous $f$-strongly slicing and $g$-strongly slicing, which exists by Theorem~\ref{th:ssGdelta}. Given $\varepsilon>0$, by using the uniform continuity of $*$ on $f(A)$ we can find $\delta>0$ such that $\max\{ \diam(U), \diam(V)\}<\delta$  for $U,V \subset f(A)$  implies that $\diam(U * V) < \varepsilon$\@. Thus, if $H \subset X$ is a half-space such that $A \cap H \neq \emptyset$ with $\diam(f(A \cap H)) <\delta$ and $\diam(g(A \cap H)) <\delta$, then $\diam((f*g)(A \cap H)) <\varepsilon$\@.
\end{proof}

Now we can give the proof of Theorem~\ref{th:structure}.

\begin{proof}[Proof of Theorem~\ref{th:structure}.] Lemma~\ref{lemma:binop} yields the first statement. Now, assume that $C$ is also bounded. Note that the boundedness and convexity of $C$ together the uniform continuity implies that every map in ${\mathcal D_U}(C,M)$ is bounded, so we may consider the uniform metric on this set.

Let $f\colon C \rightarrow M$ be a map that can be uniformly approximated by maps from ${\mathcal D_U}(C,M)$\@. Clearly, $f$ is uniformly continuous. We will see that $f$ is moreover dentable. Indeed, fix $\varepsilon>0$ and take $g \in {\mathcal D_U}(C,M)$ such that $d_\infty(f,g) <\varepsilon/3$\@. If $A \subset C$ is nonempty, then there is $H \subset X$ a half-space such that $\diam(g(A \cap H)) <\varepsilon/3$ and $A \cap H \not =\emptyset$\@. The triangle inequality yields that $\diam(f(A \cap H)) <\varepsilon$\@.

Thus, ${\mathcal D_U}(C,M)$  is closed for uniform convergence, and therefore, if $M$ is complete, then ${\mathcal D_U}(C,M)$ is complete too. From what we have proved follows that $\mathcal{D}_U(C,M)$ is a Banach space wherever $M$ is. Finally, the last statement is a straightforward consequence of Lemma~\ref{lemma:binop}.
\end{proof}

Let us remark that Schachermayer proved in~\cite{S85} that there exist sets $C_1$ and $C_2$ with the RNP such that $C_1+C_2$ contains an isometric copy of the closed unit ball of $c_0$ and thus $C_1+C_2$ does not have the RNP. This implies that the sum of two strong Radon-Nikod\'ym operators need not to be strong Radon-Nikod\'ym (an operator is said to be strong Radon-Nikod\'ym if the image of the closed unit ball has the RNP).

We finish the section by showing that uniformly continuous dentable maps satisfy a mixing property analogous to the one of $\mathcal{DC}$ functions (see~\cite[Lemma 4.8]{VZ89}). We will need the following elementary lemma.

\begin{lemma}\label{lemma:connected} Let $A$ be a connected space and let $\rho$ be a pseudometric on $A$\@. Assume that there exist closed subsets $A_1,\dotsc, A_n$ of $M$ such that $A=\cup_{i=1}^n A_i$ and $\diam(A_i)\leq \varepsilon$ for each $i$\@. Then $\diam(A) \leq n\varepsilon$\@.
\end{lemma}

\begin{proof}
Fix $x,y\in A$\@. Take $x_1=x$ and let $\sigma(1)\in\{1,\dotsc,n\}$ be such that $x\in A_{\sigma(1)}$\@. Since $A$ is connected, there is $x_2\in A_{\sigma(1)}\cap (\cup_{i\neq \sigma(1)} A_i)$, and so $\rho(x_2,x)\leq \varepsilon$\@. Take $\sigma(2)\neq \sigma(1)$ such that $x_2\in A_{\sigma(2)}$\@. Now take $x_3 \in (A_{\sigma(1)}\cup A_{\sigma(2)})\cap (\cup_{i\in \{1,\dotsc,n\}\setminus \{\sigma(1),\sigma(2)\}} A_i)$\@. Then either $\rho(x_3,x_2)\leq \varepsilon$ or $\rho(x_3, x)\leq \varepsilon$, so in any case $\rho(x_3,x)\leq 2\varepsilon$\@. By iterating this process we get $\sigma(i)$ and $x_{i}$ for each $i=1,\dotsc, n$ satisfying that $\sigma(i)\in\{1,\dotsc,n\}\setminus \{\sigma(1),\dotsc,\sigma(i-1)\}$, $x_i\in A_{\sigma(i)}\cap (\cup_{j< i} A_{\sigma(j)})$ and $\rho(x_i,x)\leq (i-1)\varepsilon$\@. Thus $\sigma$ defines a bijection on $\{1,\dotsc,n\}$ and so there exists $i$ such that $y\in A_{\sigma(i)}$\@. Therefore, $\rho(x,y)\leq \rho(y, x_i)+\rho(x_i, x)\leq n\varepsilon$, as desired. 
\end{proof}

\begin{proposition}\label{prop:mixing} Assume that $f_1,\dotsc, f_n \in \mathcal{D}_U(C, M)$ and $f\colon C\to M$ is a continuous map such that $f(x)\in \{f_1(x),\dotsc, f_n(x)\}$ for every $x\in C$\@. Then $f\in \mathcal{D}_U(C, M)$\@.
\end{proposition}

\begin{proof}
First notice that $f$ is uniformly continuous on bounded sets. Indeed, let $A\subset C$ be bounded, fix $\varepsilon>0$ and take $\delta>0$ so that $d(f_i(x)-f_i(y))\leq n^{-1}\varepsilon$ for every $i=1,\dotsc, n$ whenever $x,y\in A$ and $\norm{x-y}\leq\delta$\@.  Consider the closed sets $A_i = \{x\in A : f(x) = f_i(x)\}$\@. Now, if $x,y\in A$ satisfy that $\norm{x-y}\leq \delta$ then $\diam(A_i\cap[x,y])\leq \delta$ and thus $\diam(f(A_i\cap[x,y]))\leq  n^{-1}\varepsilon$ for every $i=1,\dotsc, n$\@. Now apply Lemma~\ref{lemma:connected} to the connected set $[x,y] = \cup_{i=1}^n A_i\cap[x,y]$ with the pseudometric $\rho=d\circ f$ to get that $d(f(x),f(y))\leq \varepsilon$\@.

In order to show that $f$ is dentable, take a nonempty bounded subset $A$ of $C$ and fix $\varepsilon>0$\@. By Proposition~\ref{prop:equident} there is $H\in \mathbb H$ satisfying that $\max\{\diam(f_i (A\cap H)): i=1,\dotsc, n\} \leq n^{-1}\varepsilon$\@. The results follows by applying Lemma~\ref{lemma:connected} to $A\cap H = \cup_{i=1}^n A_i\cap A\cap H$\@. 
\end{proof}

\section{Characterisations of dentability}\label{sec:3}

We will begin the section by showing the relation between the dentability of a set and the dentability of maps defined on it. The proof of our result is based on the following theorem which goes back to Huff and Morris~\cite{HM76} (see also~\cite{GOOT04} for a more general version):  a set $D$ is dentable if and only if it has open slices whose Kuratowski index of non-compactness is arbitrarily small. Let us recall that the Kuratowski index of non-compactness of a set $D\subset X$ is given by
\[ \alpha(D) = \inf\{\varepsilon>0 : \exists x_1,\dotsc, x_n \in X, D\subset \bigcup_{i=1}^n B(x_i, \varepsilon)\}. \]

We denote by $\omega^{<\omega}$ the set of finite sequences of natural numbers. The length of a finite sequence $s$ is denoted by $|s|$\@. Given $s,t\in \omega^{<\omega}$, we denote by $s\frown t$ the concatenation of $s$ and $t$\@.

\begin{proposition}\label{prop:charRNPdent}
Let $C\subset X$ be a closed convex set. Then the following are equivalent:
\begin{enumerate}[(i)]
\item the set $C$ has the RNP;
\item for every metric space $(M,d)$, every continuous map $f\colon C\to M$ is dentable;
\item every Lipschitz function $f\colon C\to \mathbb R$ is dentable.
\end{enumerate}
\end{proposition}

\begin{proof}
Notice that $C$ is a complete metric space as being a closed subset of $X$\@. Moreover, if $C$ has the RNP then the identity $\mathbb I\colon C\to C$ is dentable. Therefore $(i)\Rightarrow (ii)$ follows from Corollary~\ref{cor:composition}.  Moreover, clearly $(ii)$ implies $(iii)$\@. Now, we use an argument from~\cite{C98} to prove $(iii)\Rightarrow (ii)$. Assume that there is a bounded subset $A$ of $C$ which is not dentable. Then there exists $\varepsilon>0$ such that any open slice of $A$ has Kuratowski's index of non-compactness greater than $\varepsilon$\@. We will define a tree $T\subset \omega^{<\omega}$ and sequences $(x_s)_s\subset A$, $(\lambda_s)\subset [0,1]$ and $(n_s)_s\in \mathbb N$ indexed in $T$ satisfying
\begin{enumerate}
\item[1)] $\sum_{n=1}^{n_s} \lambda_{s\frown n} =1$,
\item[2)] $\norm{x_s - \sum_{n=1}^{n_s} \lambda_{s\frown n} x_{s\frown n}} < \frac{\varepsilon}{2^{|s|+2}}$, and
\item[3)] $\norm{x_t - x_{s}}\geq \varepsilon$ for each $t$ such that $|t|< |s|$
\end{enumerate}
for each $s\in T$\@. In order to show that, choose $x_\emptyset \in A$ arbitrarily. Now, assume $x_s$ has been previously defined. Then $x_s\in \overline{\conv}(A\setminus \cup_{|t|< |s|} B(x_t, \varepsilon))$\@. Indeed, otherwise there would be an open halfspace $H$ satisfying $x_s\in A\cap H \subset \cup_{|t| < |s|} B(x_t, \varepsilon)$ and thus $\alpha(A\cap H)<\varepsilon$, a contradiction. Thus, there exists $n_s\in \mathbb N$, $\lambda_{s\frown n}\geq 0$ and $x_{s\frown n}\in A$ satisfying above conditions.

In addition, a standard argument (see for instance Lemma 5.10 in~\cite{BL00}) provides sequences $(y_s)_s\subset A$, $(\mu_s)\subset [0,1]$ and $(m_s)_s$ satisfying
\begin{enumerate}
\item[1')] $\sum_{m=1}^{m_s} \mu_{s\frown m} =1$,
\item[2')] $y_s = \sum_{m=1}^{m_s} \mu_{s\frown m} y_{s\frown m}$, and
\item[3')] $\norm{x_s - y_s}\leq \varepsilon/4$ 
\end{enumerate}
for each $s\in T$\@. Thus,
\begin{equation}\label{eq:ineq}
\norm{y_t- y_{s\frown n}}\geq \norm{x_t - x_{s\frown n}}-\norm{x_t-y_t}-\norm{x_{s\frown n}-y_{s\frown n}}\geq \varepsilon/2
\end{equation}
whenever $t\leq s$\@.

Finally, take
\begin{align*}
 O &= \{y_{s} : |s| \text{ is odd}\},  &  E &= \{y_{s} : |s| \text{ is even}\}
\end{align*}
and notice that~\ref{eq:ineq} implies that $d(O, E)\geq \varepsilon/2$\@. Consider the function $f\colon C\to \mathbb R$ given by $f(x)=d(x, O)$, which is a Lipschitz function. Then $f$ is not dentable. Indeed, take $H\in \mathbb H$ satisfying $A\cap H\neq \emptyset$ and fix some $y_s\in A\cap H$\@. By condition 2' above, there is $m\leq m_s$ such that $y_{s\frown m}\in H$\@. Since either $y_s\in O, y_{s\frown m} \in E$ or $y_s\in E, y_{s\frown m} \in O$, it follows
\[ \diam(f(O\cap H))\geq |d(y_s, O)-d(y_{s\frown m}, O)|\geq \varepsilon/2, \]
as we want.
\end{proof}

Next we study finitely dentable maps, which were introduced by the second named author in~\cite{R08}. Let us recall the definition. For any dentable map $f\colon C \rightarrow M$ defined on a bounded closed convex set we may consider the following ``derivation''
\[ [D]'_{\varepsilon} = \{ x \in D: \diam(f(D \cap H))>\varepsilon, \ \forall H \in {\mathbb H}, x \in H \}\,. \]
Consider the sequence of sets indexed by ordinals defined inductively by
\[ [C]_{\varepsilon}^{\alpha+1}=[[C]_{\varepsilon}^{\alpha}]'_{\varepsilon}\]
and by $[C]_{\varepsilon}^{\alpha}= \bigcap_{\beta<\alpha}[C]_{\varepsilon}^{\beta}$ if $\alpha$ is a limit ordinal.
The dentability of $f$ implies that $[C]_{\varepsilon}^{\alpha}=\emptyset$ for some ordinal, so we may consider the ordinal index
\[ \Dz(f,\varepsilon) = \inf \{ \alpha: [C]_{\varepsilon}^{\alpha}=\emptyset \}\,. \]
We say that $f$ is \emph{finitely dentable} if $\Dz(f,\varepsilon) < \omega$ for every $\varepsilon>0$, and we
say that $f$ is \emph{countably dentable} if $\Dz(f,\varepsilon) < \omega_1$ for every $\varepsilon>0$\@.
Note that if $C$ is separable then any dentable map defined on it is countably dentable.

Let us mention that any slice $H$ which does not meet $[D]'_\varepsilon$ satisfies $\diam(f(D \cap H))\leq 2\varepsilon$, that is, Lancien's midpoint argument (see, e.g.~\cite{Lancien93}) applies also in the non-linear context. Indeed, assume that $x,y\in D$ satisfies that the segment $[x,y]$ does not meet $[D]_\varepsilon'$\@. Then we can consider the sets
\begin{align*}
A &= \{ z\in [x,y] : \exists H\in \mathbb H, [x,z]\subset H, \diam(f(D \cap H))\leq \varepsilon\},\\
B &= \{ z\in [x,y] : \exists H\in \mathbb H, [z,y]\subset H, \diam(f(D \cap H))\leq \varepsilon\}.
\end{align*}
Note that $A$ and $B$ are relatively open in $[x,y]$. A connectedness argument provides a point $z\in A\cap B$, and thus $d(f(x),f(y))\leq 2\varepsilon$\@. This fact was used repeatedly  in~\cite{R08} to provide a characterisation of finitely dentable Lipschitz maps in terms of a renorming of the Banach space. Indeed, a slight modification of the proof of that theorem shows that the same result holds for unifomly continuous maps. This result will be useful in the next section.

\begin{proposition}\label{prop:renormfindent} Let $f\colon C\to M$ be a uniformly continuous map defined on a bounded closed convex set. Then the following are equivalent:
\begin{enumerate}[(i)]
\item the map $f$ is finitely dentable;
\item there exists an equivalent norm $\eqnorm{\cdot}$ on $X$ satisfying $\lim_n d(f(x_n),f(y_n))=0$ whenever the sequences $(x_n),(y_n)\subset C$ are such that
\[ \lim_{n\to \infty} 2\eqnorm{x_n}^2+2\eqnorm{y_n}^2 - \eqnorm{x_n+y_n}^2 = 0. \]
\end{enumerate}
\end{proposition}

\begin{proof}
Let $f$ be a finitely dentable uniformly continuous map. For each $\varepsilon>0$ consider
\[ \delta(\varepsilon) = \inf\{ \norm{x-y} : x,y\in C, d(f(x),f(y))\geq \varepsilon\} \]
Let $N_k = \Dz(f,2^{-k})$ and consider the $2$-Lipschitz symmetric convex function $F$ defined on $X$ by the formula
\[ F(x)^2 = \sum_{k=1}^\infty \sum_{n=1}^{N_k} \frac{2^{-k}}{N_k}d(x,[C]_{2^{-k}}^n)^2+  \sum_{k=1}^\infty \sum_{n=1}^{N_k} \frac{2^{-k}}{N_k}d(x,-[C]_{2^{-k}}^n)^2. \]
Now go through the same steps as in the proof of~\cite[Theorem 2.2]{R08} to get that if $x,y\in C$ and $d(f(x),f(y))>\varepsilon$, then
\[ F\left(\frac{x+y}{2}\right)^2 \leq \frac{F(x)^2+F(y)^2}{2} - \frac{\varepsilon\delta(\varepsilon/4)^2}{128 \Dz(f,\varepsilon/8)^3}. \]
Therefore, an equivalent norm defined as in the proof of~\cite[Theorem 2.2]{R08} does the work.

Conversely, assume that $\eqnorm{\cdot}$ is an equivalent norm satisfying the property in the statement. We may assume that $0\in C$\@. Take $M=\sup\{\eqnorm{x}:x\in C\}$ and fix $\varepsilon>0$\@. It is not difficult to show that there exists $\delta>0$ such that $\diam(f(B_{\eqnorm{\cdot}}(0, r+\delta)\cap H)< \varepsilon$ whenever $H\in \mathbb H$ does not intersect $B_{\eqnorm{\cdot}}(0, r)$\@. Thus, $[C]_\varepsilon^n \subset B_{\eqnorm{\cdot}} (0, M-n\delta)$ for each $n$, so $\Dz(f,\varepsilon)\leq M\delta^{-1}$\@. 
\end{proof}

We will finish the section by studying the relation between the dentability of a map $f$ with values in a normed space and the dentability of the function $\norm{f}$\@. The following corollary was inspired by the absoluteness of difference convexity (see~\cite[Theorem 2.9]{BB11}).

\begin{corollary} Assume that $f\colon C\rightarrow \mathbb R$ is uniformly continuous on bounded sets. Then $f$ is dentable if and only if $|f|$ is dentable.
\end{corollary}

\begin{proof}
It is clear that $|f|$ is dentable whenever $f$ is dentable. The converse statement follows from Proposition~\ref{prop:mixing} and the fact that $f(x)\in \{|f(x)|, -|f(x)|\}$\@. 
\end{proof}

Notice that the above result fails when the modulus is replaced by the norm for dentable maps. Indeed, the identity map $\mathbb I \colon c_0 \rightarrow c_0$ is not dentable whereas $\norm{\cdot}\circ \mathbb I = \norm{\cdot}$ is dentable as being a continuous convex function which is bounded on bounded sets. 

The next result shows that it is possible to construct such an example even if the target space is two-dimensional.

\begin{proposition} Assume that $C$ is a bounded closed convex set which does not have the RNP. Then there exists a non-dentable Lipschitz map $f\colon C\rightarrow (\mathbb R^2,\norm{\cdot}_1)$ such that $\norm{f(x)}_1=1$ for every $x\in C$\@.
\end{proposition}

\begin{proof}
Let $g\colon C\to \mathbb R$ be a non-dentable Lipschitz function, which exists by Proposition~\ref{prop:charRNPdent}. We may assume that $g$ is $1$-Lipschitz and $g(0)=0$\@. Then the function given by
\[f(x)=(\diam(C)^{-1}g(x),1-\diam(C)^{-1}g(x))\]
does the work. 
\end{proof}

\section{Relation with $\mathcal{DC}$ functions and maps}

A function $f\colon C\to \mathbb R$ is said to be $\mathcal {DC}$ (or delta-convex) if it can be represented as the difference of two convex continuous functions on $C$, and it is said to be \emph{$\mathcal {DC}$-Lipschitz} (resp.~\emph{$\mathcal{DC}$-bounded}) if it is the difference of two convex Lipschitz (resp.~bounded) functions. There is a large literature about $\mathcal {DC}$ functions and its applications in analysis and optimization. Indeed, optimization problems involving $\mathcal{DC}$ functions appear frequently in engineering, economics and other sciences.  The reader is referred to the surveys~\cite{BB11,H85,T95} for relevant results, examples and applications of $\mathcal {DC}$ functions.

Let us remark that the space of $\mathcal {DC}$ functions on a convex set is not closed. Indeed, every continuous function on a norm-compact set can be uniformly approximated by $\mathcal {DC}$ functions, as a consequence of Stone\textendash{}Weierstrass theorem. On the other hand, it was shown in~\cite{Frontisi95} that for every infinite-dimensional Banach space there is a  continuous function defined on the unit ball which can not be uniformly approximated by $\mathcal {DC}$ functions uniformly on the unit ball. The situation is different if we restrict our attention to Lipschitz functions. Cepedello Boiso (\cite{C98}, see also~\cite[Theorem 5.1.25]{BW10} and~\cite[Theorem 4.21]{BL00}) characterised superreflexive spaces as those Banach spaces in which every Lipschitz function defined on it can be approximated uniformly on bounded sets by $\mathcal{DC}$ functions which are Lipschitz on bounded sets. That result was extended by the second named author in~\cite{R08} by showing that a Lipschitz function defined on a bounded closed convex set is finitely dentable if and only if is uniform limit of $\mathcal {DC}$-Lipschitz functions.  Moreover, in~\cite{CZ14} it is studied approximation by $\mathcal{DC}$ functions on a compact subset of a locally convex space, via the Kakutani\textendash{}Krein theorem.

We need to recall some definitions. A subset $D \subset X$ is said to be a \emph{$({\mathcal C} \setminus {\mathcal C})_{\sigma}$-set}
if $D=\bigcup_{n=1}^{\infty} (A_{n} \setminus B_{n})$, where $A_{n}$ and $B_{n}$ are convex closed subsets of
$X$\@. A real function $f\colon C \rightarrow {\mathbb R}$ is said to be \emph{$({\mathcal C} \setminus {\mathcal C})_{\sigma}$-measurable}
if the sets $f^{-1}(-\infty,r)$ and $f^{-1}(r,+\infty)$ are both $({\mathcal C} \setminus {\mathcal C})_{\sigma}$ subsets of $X$ for each $r\in \mathbb R$\@.

The following result summarises the connection between dentability and approximation by $\mathcal{DC}$ functions for uniformly continuous functions.

\begin{theorem}\label{th:deltaconvex}
Let $f\colon C \rightarrow {\mathbb R}$ be a uniformly continuous function defined on a bounded closed convex set.
Consider the following  statements:
\begin{enumerate}[(i)]
\item $f$ is uniform limit of $\mathcal {DC}$-bounded functions;
\item $f$ is finitely dentable;
\item $f$ is uniform limit of $\mathcal {DC}$-Lipschitz functions;
\item $f$ is countably dentable;
\item $f$ is $({\mathcal C} \setminus {\mathcal C})_{\sigma}$-measurable;
\item $f$ is pointwise limit of $\mathcal {DC}$-Lipschitz functions.
\end{enumerate}
Then $(i) \Leftrightarrow (ii) \Leftrightarrow (iii) \Rightarrow (iv) \Rightarrow (v) \Rightarrow (vi) \not \Rightarrow (v) \not \Rightarrow (iv) \not \Rightarrow (iii)$\@.
\end{theorem}

\begin{proof} First we prove the positive statements. $(i) \Rightarrow (ii)$ follows from~\cite[Proposition 3.1]{R08} and~\cite[Proposition 5.1]{R08}. $(ii) \Rightarrow (iv)$ and $(iii)\Rightarrow (i)$ are obvious. Moreover, $(v) \Rightarrow (vi)$ is~\cite[Corollary 2.7]{R03}. In order to prove $(ii)\Rightarrow (iii)$, let $\eqnorm{\cdot}$ be the equivalent norm given by Proposition~\ref{prop:renormfindent}. Since $f$ is bounded, one can consider the sequence of functions given by
\[ f_n(x) = \inf\{ f(y)+n(2\eqnorm{x}^2+2\eqnorm{y}^2-\eqnorm{x+y}^2): y\in C\}. \]
Notice that
\[ f_n(x) = 2n\eqnorm{x}^2 - \sup\{n\eqnorm{x+y}^2-2n\eqnorm{y}^2-f(y): y\in C \} \]
and thus each $f_n$ is a $\mathcal{DC}$-Lipschitz function. Now the same arguments as in the proof of~\cite[Theorem 1.4]{R08} show that $(f_n)_n$ converges uniformly to $f$ on $C$\@.
Finally we will prove $(iv) \Rightarrow (v)$\@. If $V \subset {\mathbb R}$ is open then it is not difficult to check that
\[ f^{-1}(V) =\bigcup_{n=1}^{\infty} \bigcup_{\alpha<\Dz(f,n^{-1})}
\{x \in [C]^{\alpha}_{n^{-1}} : \exists H \in {\mathbb H} \textrm{ s.~t. } x \in H, f([C]_{n^{-1}}^{\alpha}\cap H) \subset V \}\,, \]
which is a representation of the set as countable union of $({\mathcal C} \setminus {\mathcal C})$-sets, as union of open slices of a closed convex set.

Now we turn to the negative statements. Remark 2.2 in~\cite{R03} shows that the characteristic function of the Cantor set is pointwise limit of $\mathcal{DC}$-Lipschitz functions but not $(\mathcal{C}\setminus\mathcal{C})_\sigma$-measurable and thus $(vi)\not\Rightarrow (v)$\@. Moreover, let $X$ be a separable Banach space without the RNP and let $f:B_X\to\mathbb R$ be a uniformly continuous not dentable function, which exists by Proposition~\ref{prop:charRNPdent}. As a consequence of~\cite[Theorem 1.2]{R03}, every norm open subset of $X$ is a $(\mathcal C\setminus \mathcal C)_\sigma$-set and thus $f$ is $(\mathcal C\setminus \mathcal C)_\sigma$-measurable, so $(v)\not \Rightarrow (iv)$\@. Finally, in order to show that $(iv)\not \Rightarrow (iii)$, let $X$ be a separable non-superreflexive Banach space with the RNP, e.g.~$X=\ell_1$\@. By Cepedello's theorem, there exists a Lipschitz function $f\colon B_X\to \mathbb R$ which is not finitely dentable. However, Proposition~\ref{prop:charRNPdent} and the separability of $X$ imply that $f$ is countably dentable. 
\end{proof}

The definition of $\mathcal{DC}$ function was extended to the vector-valued setting by Vesel\'y and Zaj\'i\v{c}ek in~\cite{VZ89} (see also~\cite{VZ16}). A continuous map $F\colon C\to Y$ defined on a convex subset $C\subset X$ is said to be a \emph{$\mathcal{DC}$ map} if there exists a continuous (necesarily convex) function $f$ on $C$ such that $f+y^*\circ F$ is a convex continuous function on $C$ for every $y^*\in S_{Y^*}$\@.  The function $f$ is called a \emph{control function} for $F$\@. Let us notice that in Proposition 1.13 in~\cite{VZ89} it is showed that $f$ is a control function for $F$ if, and only if,
\[ \norm{\sum_{i=1}^n \lambda_i F(x_i) - F(\sum_{i=1}^n \lambda_i x_i)}\leq \sum_{i=1}^n \lambda_i f(x_i) - f(\sum_{i=1}^n \lambda_i x_i)\]
whenever $x_1,\dotsc, x_n\in A$, $\lambda_1,\dotsc,\lambda_n\geq 0$ and $\sum_{i=1}^n \lambda_i = 1$\@.

We say that a $\mathcal{DC}$ map is \emph{$\mathcal{DC}$-bounded} if it is bounded and admits a bounded control function. The space of $\mathcal{DC}$-bounded maps  has been recently studied in~\cite{VZ16}. In the next result we characterise the dentability of such maps in terms of the dentability of their target space. We will denote by $\mathbb E(f|\mathcal F)$ the conditional expectation of a function $f$ given a $\sigma$-algebra $\mathcal F$\@.

\begin{proposition} Let $D\subset Y$ be a closed convex set. Then the following are equivalent:
\begin{enumerate}[(i)]
\item the set $D$ has the RNP;
\item for every Banach space $X$ and every convex subset $C\subset X$, every $\mathcal{DC}$-bounded map $F\colon C\to D$ is dentable.
\end{enumerate}
\end{proposition}
\begin{proof}
First, assume that $D$ does not have the RNP. Let $A\subset D$ be a non-dentable bounded subset of $D$ and take $C=\overline{\conv}(D)$\@. Then the identity $\mathbb I\colon C\to D$ is a non-dentable bounded continuous $\mathcal {DC}$ map with the zero function as control function. Thus $(ii)$ implies $(i)$\@.

Now take a bounded continuous $\mathcal{DC}$ map $F\colon C\to D$ with a bounded control function $f$\@.  Given $\varepsilon>0$ and $x\in C$, we will denote
\[ \delta(x, \varepsilon) = \inf\{\norm{x-y}: y\in C, \max\{\norm{F(x)-F(y)},|f(x)-f(y)|\}>\varepsilon\}\,.\]
Assume that $F$ is not dentable. Then by Proposition~\ref{prop:dent} there exist $\varepsilon>0$ and $A\subset C$ bounded such that $x\in \overline{\conv}(A\setminus F^{-1}(B_Y(F(x),\varepsilon))$ for each $x\in A$\@.  We will construct a martingale $(h_n)_n$ with values in $F(A)\subset D$ and so that $\norm{h_n-h_{n+1}}\geq \varepsilon/2$\@.  In order to do that, we will define inductively an increasing sequence $(\mathcal F_n)_n$ of $\sigma$-algebras in the interval $[0,1]$ and a sequence $(g_n)_n$ of functions from $[0,1]$ to $A$ satisfying the following conditions for each $n\in \mathbb N$:
\begin{enumerate}
\item[1)] $\mathcal F_n$ is the $\sigma$-algebra generated by a finite partition $\pi_n$ of the unit interval into disjoint subintervals;
\item[2)] $g_n$ is $\mathcal F_n$-measurable;
\item[3)] $\norm{F (g_n(t))-F(g_{n+1}(t))}\geq \varepsilon$ for each $n\in\mathbb N$ and $t\in [0,1]$;
\item[4)] $\norm{F\circ g_n-\mathbb E(F\circ g_{n+1}|\mathcal F_n)}_1 \leq \int_0^1 (f\circ g_{n+1}-f\circ g_n)d\m + \frac{\varepsilon}{2^{n+3}}$, where $\m$ denotes the Lebesgue measure on $[0,1]$\@.
\end{enumerate}
Fix $x_0\in A$ so that $f(x_0)\geq \sup\{f,A\}-\frac{\varepsilon}{16}$, take $\pi_0=\{[0,1]\}$ and define $g_0(t)=x_0$ for each $t\in [0,1]$\@.  Assume we have defined $\pi_n = \{A_1,\dotsc, A_p\}$, $\mathcal F_n = \sigma(\pi_n)$ and a $\mathcal F_n$-measurable map $g_n\colon [0,1]\to B$\@.  Then there exist $x_1,\dotsc, x_p\in A$ such that $g_n = \sum_{i=1}^p x_i \chi_{A_i}$\@.  Now, the non-dentability of $F$ on $A$ implies that there are integers $k_i$, $\lambda_{ij} \geq 0$, and $x_{ij} \in A\setminus F^{-1}(B_Y(F(x_i),\varepsilon))$ for $j\in\{1,\dotsc, k_i\}$ satisfying that $\sum_{j=1}^{k_i} \lambda_{ij}=1$ and $\norm{\sum_{j=1}^{k_i} \lambda_{ij}x_{ij} - x_i}\leq \delta(x_i, 2^{-n-5}\varepsilon)$\@.  For each $i$, let $\{A_{ij} \}$ be a partition of $A_i$ into disjoint subintervals with $\m(A_{ij})=\lambda_{ij}\m(A_i)$\@.  Take $\pi_{n+1} = \{A_{ij}\}_{ij}$ and $\mathcal F_{n+1} = \sigma(\pi_{n+1})$\@.  Finally, define $g_{n+1} = \sum_{i=1}^p\sum_{k=1}^{k_i} x_{ij} \chi_{A_{ij}}$\@.  Clearly, $g_{n+1}$ is $\mathcal F_{n+1}$-measurable and takes values on $F(A)$\@.  Moreover, 
\[\norm{F(g_n(t))-F(g_{n+1}(t))}\geq \varepsilon\]
for each $t\in [0,1]$ since $x_{ij} \in B\setminus F^{-1}(B_Y(F(x_i),\varepsilon)$ for each $i,j$\@. By using the fact that $f$ is a control function for $F$ and the definition of $\delta$ we get that
\begin{align*}
\norm{\sum_{j=1}^{k_i}\lambda_{ij} F(x_{ij})-F(x_i)} & \leq \norm{\sum_{j=1}^{k_i}\lambda_{ij} F(x_{ij})-F(\sum_{j=1}^{k_i}\lambda_{ij} x_{ij})}+ \frac{\varepsilon}{2^{n+5}}\\
& \leq \left(\sum_{j=1}^{k_i}\lambda_{ij} f(x_{ij})-f(\sum_{j=1}^{k_i}\lambda_{ij}x_{ij})\right) + \frac{\varepsilon}{2^{n+5}} \\
& \leq \left(\sum_{j=1}^{k_i}\lambda_{ij} f(x_{ij})-f(x_i)\right) + \frac{\varepsilon}{2^{n+4}}\,.
\end{align*}
In addition, it is easy to show that $\mathbb E(F\circ g_{n+1}|\mathcal F_n) = \sum_{i=1}^n (\sum_{j=1}^{k_i} \lambda_{ij} F(x_{ij})) \chi_{A_{i}}$\@.  Thus, we can estimate $\norm{F\circ g_{n}-\mathbb E(F\circ g_{n+1}|\mathcal F_n)}_1$ as follows:
\allowdisplaybreaks
\begin{align*}
\norm{F\circ g_{n}-\mathbb E(F\circ g_{n+1}|\mathcal F_n)}_1 &\leq \sum_{i=1}^p \m(A_i)\norm{F(x_i)-\sum_{j=1}^{k_i}\lambda_{ij} F(x_{ij})}\\
&\leq \sum_{i=1}^p \m(A_i) \left(\sum_{j=1}^{k_i}\lambda_{ij} f(x_{ij})-f(x_i)\right) + \frac{\varepsilon}{2^{n+4}}\\
&= \sum_{i=1}^p\sum_{j=1}^{k_i} \m(A_{ij}) f(x_{ij}) - \sum_{i=1}^p \m(A_i) f(x_i) + \frac{\varepsilon}{2^{n+4}}\\
&= \int_0^1 f\circ g_{n+1}d\m -\int_0^1 f\circ g_n d\m + \frac{\varepsilon}{2^{n+4}}\,.
\end{align*}
This shows that conditions $1)$ to $4)$ above are satisfied. Finally, from condition $4)$ and the fact that $f(x_0)\geq \sup\{f,A\}-\frac{\varepsilon}{16}$ we get that
\[ \sum_{n=0}^N \norm{F\circ g_{n}-\mathbb E(F\circ g_{n+1}|\mathcal F_n)}_1 \leq \int_0^1 (f\circ g_{N+1} - f\circ g_0)d\m + \sum_{n=0}^N \frac{\varepsilon}{2^{n+4}} \leq \frac{\varepsilon}{16} + \frac{\varepsilon}{8}\,\]
for each $N\in\mathbb N$\@.  It follows that
\[ \sum_{n=0}^\infty \norm{F\circ g_{n}-\mathbb E(F\circ g_{n+1}|\mathcal F_n)}_1 < \frac{\varepsilon}{4}\,.\]
Thus, one can apply Lemma 5.10 in~\cite{BL00} to get a martingale $(h_n)_n$ with values in the bounded set $F(A)\subset D$ so that $\norm{h_n-F\circ g_n}_1\leq \varepsilon/4$\@.  Therefore
\[ \norm{h_n-h_{n+1}}_1\geq \norm{F\circ g_n- F\circ g_{n+1}}_1 - \norm{h_n- F\circ g_n}_1-\norm{h_{n+1}-F\circ g_{n+1}}_1\geq \frac{\varepsilon}{2}\,,\]
which contradicts the assumption that $D$ has the Radon-Nikod\'ym property. 
\end{proof}

\section{Dentable sets with respect to a metric}\label{sec:4}

Notice that, given a uniformly continuous injective map $f\colon C\to (M,d)$, we have that $f$ is dentable if and only if $C$ is dentable with respect to the metric $\rho:= d\circ f$. Moreover, $\rho$ is a uniformly continuous metric on $C$, that is, for every $\varepsilon>0$ there exists $\delta>0$ so that $d(x,y)<\varepsilon$ whenever $x,y\in C$, $\norm{x-y}<\delta$\@.

\begin{proposition}\label{prop:sepDent}
If $C \subset X$ is a bounded closed convex subset that admits a separating sequence, then it is dentable with respect to some norm-Lipschitz metric defined on it.
\end{proposition}

\begin{proof} Fix $(x_n^*) \subset B_{X^*}$ a sequence which is separating on $C$ and define
\[ d(x,y)=\sum_{n=1}^{\infty} 2^{-n}|x_n^*(x-y)|.\]
Clearly, it is a metric on $C$ and it is norm Lipschitz. The dentability of $C$ with respect to $d$ follows from Corollary~\ref{cor:equidual}.
Indeed, fix $A \subset C$ nonempty and $\varepsilon>0$\@. Take $n \in {\mathbb N}$ such that $2^{-n} < \varepsilon$ and find a slice $S \subset C$ such that $|x^*_k(x-y)| < \varepsilon/2$ for all $k \leq n$ and $x,y \in S$\@. Then $\diam(S)<\varepsilon$ with respect to $d$\@.
\end{proof}

Let us recall that Stegall's variational principle (see, e.g.~\cite[Theorem 5.17]{BL00}) states that a lower semicontinuous bounded below function $f$ defined on a set with the RNP can be modified by adding an arbitrarily small linear perturbation $x^*$ in such a way that the resulting function $f + x^*$ admits a strong minimum. Geometrically that is the consequence of the existence of strongly exposed points of the epigraph of $f$ and of many strongly exposing functionals, as $x^*$\@. The pathological fact in our frame is that the slices associated to a strongly slicing functional could be convergent to a non strongly exposed point. Indeed, consider the following example. Let $C= B_{X}$, where $X=Z^*$ and $Z$ is a separable Banach space. By Proposition~\ref{prop:sepDent}, $C$ is dentable for a metric $d$ compatible with the weak$^*$ topology induced by the convergence on $Z$\@. Moreover, $d$ is a complete metric since $C$ is weak$^*$-compact. Thus, if $x^* \in \StS({\mathbb I}, C)$ then $S(A,x^*,t)$ converges to a point $x \in C$ when $t \rightarrow 0^+$\@. However, we can not ensure that $x \in \bigcap_{t>0} S(A,x^*,t)$ since in general $x^*$ is not continuous with respect to $d$\@. Adding such a hypothesis would send us back to the classical case, as Theorem~\ref{th:rnp} and Proposition~\ref{prop:dentRNP} show.

\begin{proof}[Proof of Theorem~\ref{th:rnp}.] Let $\mathbb I\colon C\rightarrow (C,d)$ denote the identity map, which by assumption is uniformly continuous on bounded sets. Given $A \subset C$ be nonempty and bounded, Theorem~\ref{th:ssGdelta} provides a $\mathbb I$-strongly slicing functional $x^*$\@. Moreover, since $d$ induces the norm topology, $\diam(\overline{S(A,x^*, t)})=\diam(S(A,x^*, t))$ for each $t>0$\@. Thus, the completeness of $d$ implies that $\bigcap_{t>0} \overline{S(A,x^*, t)}$ consists exactly on one point $y\in C$\@. Given $\varepsilon>0$, the coincidence of the norm topology and the one induced by $d$ provides $\delta>0$ such that $B_d(y,\delta)\subset B_{\norm{\cdot}}(y,\varepsilon)$\@. Now, if $t>0$ is small enough then $S(A,x^*, t)$ is contained in $B_d(y,\delta)$ and therefore the norm-diameter of $S(A,x^*, t)$ is less than $2\varepsilon$\@. Thus $A$ is norm-dentable and we are done.
\end{proof}

Let us remark that it is possible to relax the hypothesis of coincidence of topologies in the previous result. 

\begin{proposition}\label{prop:dentRNP} Assume that $C$ is dentable with respect to a metric $d$ which is uniformly continuous on bounded sets with respect to the norm. Consider $\mathbb I\colon C\to (C,d)$ the identity map. If for every $A\subset C$ nonempty closed convex bounded and every $x^*\in \StS(\mathbb{I},A)$ the set $\bigcap_{t>0} \overline{S(A,x^*,t)}^{\norm{\cdot}}$ is nonempty, then $C$ has the RNP.   
\end{proposition}

\begin{proof}
Let $A\subset C$ be nonempty closed convex bounded. Given $x^*\in \StS(\mathbb{I},A)$, take $x\in\bigcap_{t>0} \overline{S(A,x^*,t)}^{\norm{\cdot}}$\@. Then $x^*(x)\geq \sup\{x^*,A\}-t$ for each $t>0$ and thus $x^*$ supports $A$ at $x$\@. It follows that the set of support functionals of $A$ contains a dense ${\mathcal G}_\delta$ in $X^*$\@. The Bourgain\textendash{}Stegall theorem, see~\cite[Theorem 3.5.5]{B83}, implies that $C$ has the RNP.
\end{proof}

Finally, the following proposition provides an example of a bounded closed convex set which is not dentable with respect to any translation invariant metric. Let us recall that a theorem of Kakutani (see e.g.~\cite{R85}) asserts that a metric on a vector space for which addition and scalar multiplication are continuous is equivalent to a translation-invariant metric.

\begin{proposition} Let $C$ be the unit closed ball of $c_0(\Gamma)$, where $\Gamma$ is uncountable, and let $d$ be a translation-invariant metric on $C$\@. Then $C$ is not dentable with respect to $d$\@.
\end{proposition}

\begin{proof}
Denote by $\{e_\gamma\}_{\gamma\in \Gamma}$ the standard basis of $c_0(\Gamma)$\@. Fix $\varepsilon>0$ and take $x^*=(x_\gamma^*)_{\gamma\in \Gamma}\in S_{\ell_1(\Gamma)}$ providing a slice $S=\{x\in C: x^*(x)>t\}$ with $\diam(S)<\varepsilon$\@. Fix a point $x\in S$\@. Then there exists a finite subset $A_\varepsilon \subset \Gamma$ such that $\sum_{\gamma\notin A_\varepsilon} |x_\gamma^*| < x^*(x)-t$\@. Thus, given $\eta\in \Gamma\setminus A_\varepsilon$ we have
\[ x^*(x\pm e_\eta) \geq x^*(x) - |x_\eta^*| > t\]
and so $x\pm e_\eta\in S$\@. Therefore, $d(e_\eta, -e_\eta) = d(x+e_\eta, x-e_\eta) \leq \varepsilon$ for every $\eta\in \Gamma \setminus A_\varepsilon$\@.
Now, take $\eta \in \Gamma\setminus \cup_n A_{1/n}$, which exists since $\Gamma$ is uncountable. Then $d(e_\eta, -e_\eta)<1/n$ for every $n$, which is a contradiction.
\end{proof}

Note that such a phenomenon is impossible for a bounded closed convex subset $C$ of $\ell_\infty(\mathbb N)$ as a consequence of Proposition~\ref{prop:sepDent}.

\bibliographystyle{siam}
\bibliography{biblio}

\end{document}